\documentclass{amsart}
\usepackage{amssymb, amsmath, mathrsfs, verbatim, url, bbm}

\theoremstyle{plain}
\newtheorem{theorem}{Theorem}
\newtheorem{lemma}[theorem]{Lemma}
\newtheorem{proposition}[theorem]{Proposition}
\newtheorem{corollary}[theorem]{Corollary}

\theoremstyle{definition}
\newtheorem{definition}[theorem]{Definition}

\theoremstyle{remark}
\newtheorem*{remark}{Remark}
\newtheorem*{question}{Question}

\DeclareMathOperator{\CO}{\textrm{Clop}}
\DeclareMathOperator{\Fn}{Fn}
\DeclareMathOperator{\Ex}{Ex}
\DeclareMathOperator{\cl}{cl}

\newcommand{\AAA}{\mathcal{A}}
\newcommand{\BBB}{\mathcal{B}}
\newcommand{\UUU}{\mathcal{U}}
\newcommand{\VVV}{\mathcal{V}}
\newcommand{\RRR}{\mathbb{R}}

\begin{document}

\title{Products and h-homogeneity}

\author{Andrea Medini}

\date{July 13, 2010}

\address{University of Wisconsin-Madison Mathematics Dept.}
\email{medini@math.wisc.edu}

\begin{abstract}
Building on work of Terada, we prove that h-homogeneity is productive in the class of zero-dimensional spaces. Then, by generalizing a result of Motorov, we show that for every non-empty zero-dimensional space $X$ there exists a non-empty zero-dimensional space $Y$ such that $X\times Y$ is h-homogeneous. Also, we simultaneously generalize results of Motorov and Terada by showing that if $X$ is a space such that the isolated points are dense then $X^\kappa$ is h-homogeneous for every infinite cardinal $\kappa$. Finally, we show that a question of Terada (whether $X^\omega$ is h-homogeneous for every zero-dimensional first-countable $X$) is equivalent to a question of Motorov (whether such an infinite power is always divisible by $2$) and give some partial answers.
\end{abstract}

\maketitle

All spaces in this paper are assumed to be Tychonoff. It is easy to see that every zero-dimensional space is Tychonoff. For all undefined topological notions, see \cite{engelking}. For all undefined Boolean algebraic notions, see \cite{koppelberg}. Recall that a subset of a space is \emph{clopen} if it is closed and open.

\begin{definition} 
A space $X$ is \emph{h-homogeneous} (or \emph{strongly homogeneous}) if all non-empty clopen subsets of $X$ are homeomorphic to each other.
\end{definition}

The Cantor set, the rationals and the irrationals are examples of h-homogeneous spaces. Every connected space is h-homogeneous. A finite space is h-homogeneous if and only if it has size at most $1$. The concept of h-homogeneity has been studied (mostly in the zero-dimensional case) by several authors: see \cite{matveev1} for an extensive list of references.

We will denote by $\CO(X)$ the Boolean algebra of the clopen subsets of $X$. Recall that a Boolean algebra $\AAA$ is \emph{homogeneous} if $\AAA\upharpoonright a$ is isomorphic to $\AAA$ for every non-zero $a\in\AAA$, where $\AAA\upharpoonright a$ denotes the \emph{relative algebra} $\{x\in\AAA :x\leq a \}$. If $X$ is h-homogeneous then $\CO(X)$ is homogeneous; the converse holds if $X$ is compact and zero-dimensional (see the remarks following Definition 9.12 in \cite{koppelberg}).

\section{The productivity of h-homogeneity}

In \cite{terada}, the productivity of h-homogeneity is stated as an open problem (see also \cite{matveev1} and \cite{matveev2}), and it is shown that the product of zero-dimensional h-homogeneous spaces is h-homogeneous provided it is compact or non-pseudocompact (see Theorem 3.3 in \cite{terada}). The following theorem, proved by Terada under the additional assumption that $X$ is zero-dimensional (see Theorem 2.4 in \cite{terada}), is the key ingredient in the proof. Recall that a collection $\BBB$ consisting of non-empty open subsets of a space $X$ is a \emph{$\pi$-base} if for every non-empty open subset $U$ of $X$ there exists $V\in\BBB$ such that $V\subseteq U$.

\begin{theorem}[Terada]\label{teradanonpseudocompact}
Assume that $X$ is non-pseudocompact. If $X$ has a $\pi$-base consisting of clopen sets that are homeomorphic to $X$ then $X$ is h-homogeneous.
\end{theorem}

The proof of Theorem \ref{teradanonpseudocompact} uses the fact that a zero-dimensional non-pseudocompact space can be written as the disjoint union of infinitely many of its non-empty clopen subsets (the converse is also true, trivially). However, that is the only consequence of zero-dimensionality that is actually used (see Appendix A). Therefore such assumption is redundant by the following lemma, whose proof we leave to the reader.

\begin{lemma}\label{dispensezerodimensionality}
Assume that $X$ is non-pseudocompact. If $X$ has a $\pi$-base consisting of clopen sets then $X$ can be written as the disjoint union of infinitely many of its non-empty clopen subsets.
\end{lemma}

Using Theorem \ref{teradanonpseudocompact} one can easily prove the following.

\begin{theorem}[Terada]\label{teradaproductive}
Assume that $X=\prod_{i\in I}X_i$ is non-pseudocompact. If $X_i$ is h-homogeneous and it has a $\pi$-base consisting of clopen sets for every $i\in I$ then $X$ is h-homogeneous.
\end{theorem}

For proofs of the following basic facts about $\beta X$, see Section 11 in \cite{vanmill1}. Given any open subset $U$ of $X$, define $\Ex(U)=\beta X\setminus\cl_{\beta X}(X\setminus U)$. It is easy to see that $\Ex(U)$ is the biggest open subset of $\beta X$ such that its intersection with $X$ is $U$. If $C$ is a clopen subset of $X$ then $\Ex(C)=\cl_{\beta X}(C)$, hence $\Ex(C)$ is clopen in $\beta X$. Furthermore, the collection $\{\Ex(U):U\textrm{ is open in }X\}$ is a base for $\beta X$.

\begin{remark}
It is not true that $\beta X$ is zero-dimensional whenever $X$ is zero-dimensional (see Example 6.2.20 in \cite{engelking} or Example 3.39 in \cite{walker}). If $\beta X$ is zero-dimensional then $X$ is called \emph{strongly zero-dimensional}.
\end{remark}

We will need the following theorem (see Theorem 8.25 in \cite{walker}); see also Exercise 3.12.20(d) in \cite{engelking}. Recall that a subspace $Y$ of $X$ is \emph{$C^\ast$-embedded} in $X$ if every bounded continuous function $f:Y\longrightarrow\RRR$ admits a continuous extension to $X$.

\begin{theorem}[Glicksberg]\label{glicksbergtheorem} Assume that $X=\prod_{i\in I}X_i$ is pseudocompact. Then $X$ is $C^\ast$-embedded in $\prod_{i\in I}\beta X_i$.
\end{theorem}

\begin{remark} The reverse implication is also true, under the additional assumption that $\prod_{j\neq i}X_j$ is infinite for every $i\in I$. Such assumption is clearly not needed in the above statement (see Proposition 8.2 in \cite{walker}).
\end{remark}

\begin{proposition}\label{clopenpseudocompact}
Assume that $X\times Y$ is pseudocompact. If $C$ is a clopen subset of $X\times Y$ then $C$ can be written as the union of finitely many open rectangles.
\end{proposition}

\begin{proof}
It follows from Theorem \ref{glicksbergtheorem} that $X\times Y$ is $C^\ast$-embedded in $\beta X\times\beta Y$. By the universal property of the \v{C}ech-Stone compactification (see Corollary 3.6.3 in \cite{engelking}), there exists a homeomorphism $h:\beta X\times\beta Y\longrightarrow\beta(X\times Y)$ such that $h(x,y)=(x,y)$ whenever $(x,y)\in X\times Y$.

Let $C$ be a clopen subset of $X\times Y$. Since $\{\Ex(U):U\textrm{ is open in }X\}$ is a base for $\beta X$ and $\{\Ex(V):V\textrm{ is open in }Y\}$ is a base for $\beta Y$, the collection
$$
\BBB=\{\Ex(U)\times\Ex(V):U\textrm{ is open in }X\textrm{ and }V\textrm{ is open in }Y\}
$$
is a base for $\beta X\times\beta Y$. Therefore $\{h[B]:B\in\BBB\}$ is a base for $\beta(X\times Y)$, hence we can write $\Ex(C)=h[B_1]\cup\cdots\cup h[B_n]$ for some $B_1,\ldots, B_n\in\BBB$ by compactness.

Finally, if we let $B_i=\Ex(U_i)\times\Ex(V_i)$ for each $i$, we get
\begin{eqnarray}\nonumber
C & = & \Ex(C)\cap (X\times Y) \\\nonumber
& = & \left(h[B_1]\cup\cdots\cup h[B_n]\right)\cap h[X\times Y]\\\nonumber
& = & h[B_1\cap (X\times Y)]\cup\cdots\cup h[B_n\cap (X\times Y)]\\\nonumber
& = & (B_1\cap (X\times Y))\cup\cdots\cup (B_n\cap (X\times Y))\\\nonumber
& = & (U_1\times V_1)\cup\cdots\cup (U_n\times V_n),
\end{eqnarray}
that concludes the proof.
\end{proof}

\begin{lemma}\label{rectangles}
Assume that $C$ is a clopen subset of $X\times Y$ that can be written as the union of finitely many rectangles. Then $C$ can be written as the union of finitely many pairwise disjoint clopen rectangles.
\end{lemma}

\begin{proof}
For every $x\in X$, let $C_x=\{y\in Y:(x,y)\in C\}$ be the corresponding vertical cross-section. For every $y\in Y$, let $C^y=\{x\in X:(x,y)\in C\}$ be the corresponding horizontal cross-section. Since $C$ is clopen, each cross-section is clopen.

Let $\AAA$ be the Boolean subalgebra of $\CO(X)$ generated by $\{C^y:y\in Y\}$. Since $\AAA$ is finite, it must be atomic. Let $P_1,\ldots, P_m$ be the atoms of $\AAA$. Similarly, let $\BBB$ be the Boolean subalgebra of $\CO(Y)$ generated by $\{C_x:x\in X\}$, and let $Q_1,\ldots, Q_n$ be the atoms of $\BBB$.

Observe that the rectangles $P_i\times Q_j$ are clopen and pairwise disjoint. Furthermore, given any $i,j$, either $P_i\times Q_j\subseteq C$ or $(P_i\times Q_j)\cap C=\varnothing$. Hence $C$ is the union of a (finite) collection of such rectangles.
\end{proof}

\begin{proposition}\label{twoproductive}
Assume that $X\times Y$ is pseudocompact. If $X$ is h-homogeneous and $Y$ is h-homogeneous then $X\times Y$ is h-homogeneous.
\end{proposition}

\begin{proof}
Assume that $X$ and $Y$ are h-homogeneous. If $X$ and $Y$ are both connected then $X\times Y$ is connected. So assume without loss of generality that $X$ is not connected. Since $X$ is also h-homogeneous, it follows that $X\cong n\times X$ whenever $1\leq n<\omega$. Therefore $X\times Y\cong n\times X\times Y$ whenever $1\leq n<\omega$.

Now let $C$ be a non-empty clopen subset of $X\times Y$. By Proposition \ref{clopenpseudocompact} and Lemma \ref{rectangles}, we can write $C$ as the disjoint union of finitely many, say $n$, non-empty clopen rectangles. By the h-homogeneity of $X$ and $Y$, each such rectangle is homeomorphic to $X\times Y$. Therefore $C\cong n\times X\times Y\cong X\times Y$.
\end{proof}

\begin{corollary}\label{finitelyproductive}
Assume that $X=X_1\times \cdots \times X_n$ is pseudocompact. If each $X_i$ is h-homogeneous then $X$ is h-homogeneous.
\end{corollary}

An obvious modification of the proof of Proposition \ref{clopenpseudocompact} yields the following.

\begin{proposition} Assume that $X=\prod_{i\in I}X_i$ is pseudocompact. If $C$ is a clopen subset of $X$ then $C$ can be written as the union of finitely many open rectangles.
\end{proposition}

\begin{corollary}\label{finitelymanycoordinates}
Assume that $X=\prod_{i\in I}X_i$ is pseudocompact. If $C$ is a clopen subset of $X$ then $C$ depends on finitely many coordinates.
\end{corollary}

\begin{remark}
The zero-dimensional case of Corollary \ref{finitelymanycoordinates} is a trivial consequence of a result by Broverman (see Theorem 2.6 in \cite{broverman}).
\end{remark}

\begin{theorem}\label{productivepseudocompact} Assume that $X=\prod_{i\in I}X_i$ is pseudocompact. If $X_i$ is h-homogeneous for every $i\in I$ then $X$ is h-homogeneous.
\end{theorem}

\begin{proof}
Assume that each $X_i$ is h-homogeneous. Let $C$ be a non-empty clopen subset of $X$. By Corollary \ref{finitelymanycoordinates}, there exists a finite subset $F$ of $I$ such that $C$ is homeomorphic to $D\times \prod_{i\in I\setminus F}X_i$, where $D$ is a non-empty clopen subset of $\prod_{i\in F}X_i$. But $\prod_{i\in F}X_i$ is h-homogeneous by Corollary \ref{finitelyproductive}, so $D\cong\prod_{i\in F}X_i$. Hence $C\cong X$.
\end{proof}

\begin{theorem}\label{productive}
If $X_i$ is h-homogeneous and it has a $\pi$-base consisting of clopen sets for every $i\in I$ then $X=\prod_{i\in I}X_i$ is h-homogeneous.
\end{theorem}

\begin{proof}
If $X$ is pseudocompact, apply Theorem \ref{productivepseudocompact}; if $X$ is non-pseudocompact, apply Theorem \ref{teradaproductive}.
\end{proof}

\begin{corollary}\label{productivezerodim}
If $X_i$ is h-homogeneous and zero-dimensional for every $i\in I$ then $\prod_{i\in I}X_i$ is h-homogeneous.
\end{corollary}

\begin{question}
Can the zero-dimensionality requirement be dropped in Corollary \ref{productivezerodim}?
\end{question}

\section{Some applications}

The compact case of the following result was essentially proved by Motorov (see Theorem 0.2(9) in \cite{motorov2} and Theorem 2 in \cite{motorov1}).

\begin{theorem}\label{motorovpibase}
Assume that $X$ has a $\pi$-base $\BBB$ consisting of clopen sets. Then $Y=(X\times 2\times \prod\BBB)^\kappa$ is h-homogeneous for every infinite cardinal $\kappa$.
\end{theorem}

\begin{proof}
One can easily check that $Y$ has a $\pi$-base consisting of clopen sets that are homeomorphic to $Y$. Therefore, if $Y$ is non-pseudocompact, the result follows from Theorem \ref{teradanonpseudocompact}.

On the other hand, an analysis of Motorov's proof shows that the only consequence of the compactness of $Y$ that is used is the fact that clopen sets in $Y$ depend on finitely many coordinates. Therefore the same proof works if $Y$ is pseudocompact by Corollary \ref{finitelymanycoordinates}. We reproduce such proof for the convenience of the reader.

Assume that $Y$ is pseudocompact and let $C$ be a non-empty clopen subset of $Y$. The fact that $C$ depends on finitely many coordinates implies that $C\cong Y\times C$. So it will be enough to show that $Y\times C\cong Y$.

Let $B$ be a clopen subset of $C$ that is homeomorphic to $Y$. Let $D=C\setminus B$ and $E=(Y\setminus C)\oplus B$. Observe that $Y\cong Y^2\cong (Y\times D)\oplus (Y\times E)$ and that $Y\oplus Y\cong 2\times Y\cong Y$. So
\begin{eqnarray}\nonumber
Y\times C & \cong & (Y\times D)\oplus (Y\times B)\\\nonumber
& \cong & (Y\times D)\oplus Y^2\\\nonumber
& \cong & (Y\times D)\oplus \left((Y\times D)\oplus (Y\times E)\right)\\\nonumber
& \cong & \left((Y\oplus Y)\times D\right)\oplus (Y\times E)\\\nonumber
& \cong & (Y\times D)\oplus (Y\times E)\\\nonumber
& \cong & Y,
\end{eqnarray}
that concludes the proof.
\end{proof}

\begin{remark}
In \cite{motorov1} and \cite{motorov2}, Motorov asked whether the $2$ can be dropped in the definition of $Y$. This is certainly true if $Y$ is non-pseudocompact, but we do not know the answer in general. Observe that if the answer were `yes' then Theorem \ref{denseisolated} would become an immediate corollary of Theorem \ref{motorovpibase}.
\end{remark}

\begin{corollary}\label{huspenskii}
For every non-empty zero-dimensional space $X$ there exists a non-empty zero-dimensional space $Y$ such that $X\times Y$ is h-homogeneous. Furthermore, if $X$ is compact, then $Y$ can be chosen to be compact.
\end{corollary}

\begin{question}
Is it true that for every non-empty space $X$ there exists a non-empty space $Y$ such that $X\times Y$ is h-homogeneous? If $X$ is compact, can $Y$ be chosen to be compact?
\end{question}

\begin{remark}
In \cite{uspenskii}, using a very brief and elegant argument, Uspenski{\u\i} proved that for every non-empty space $X$ there exists a non-empty space $Y$ such that $X\times Y$ is homogeneous (in the sense of Definition \ref{homogeneous}). However, it is not true that $Y$ can be chosen to be compact whenever $X$ is compact: Motorov proved that the closure in the plane of $\{(x,\sin(1/x)):x\in (0,1]\}$ is not the retract of any compact homogeneous space (see Section 3 in \cite{arkhangelskii} for a proof).
\end{remark}

The following was proved by Matveev (see Proposition 3 in \cite{matveev1}) under the additional assumption that $X$ is zero-dimensional, even though such assumption is not actually used in the proof (see Appendix A). Recall that a sequence $\langle A_n:n\in\omega\rangle$ of subsets of a space $X$ \emph{converges to a point} $x$ if for every neighborhood $U$ of $x$ there exists $N\in\omega$ such that $A_n\subseteq U$ for each $n\geq N$.

\begin{proposition}[Matveev]\label{matveev}
Assume that $X$ has a $\pi$-base consisting of clopen sets that are homeomorphic to $X$. If there exists a sequence $\langle U_n:n\in\omega\rangle$ of non-empty open subsets of $X$ that converges to a point then $X$ is h-homogeneous. 
\end{proposition}

The case $\kappa=\omega$ of the following result is an easy consequence of Proposition \ref{matveev}. Motorov first proved it under the additional assumption that $X$ is a zero-dimensional first-countable compact space (see Theorem 0.2(2) in \cite{motorov2} and Theorem 1 in \cite{motorov1}). Terada proved it for an arbitrary infinite $\kappa$, under the additional assumption that $X$ is zero-dimensional and non-pseudocompact (see Corollary 3.2 in \cite{terada}).

\begin{theorem}\label{denseisolated}
Assume that $X$ is a space such that the isolated points are dense. Then $X^\kappa$ is h-homogeneous for every infinite cardinal $\kappa$.
\end{theorem}

\begin{proof}
We will show that $X^\omega$ is h-homogeneous and it has a $\pi$-base consisting of clopen sets. Since $X^\kappa\cong (X^\omega)^\kappa$ for every infinite cardinal $\kappa$, an application of Theorem \ref{productive} will conclude the proof.

Let $D$ be the set of isolated points of $X$ and let $\Fn(\omega, D)$ be the set of finite partial functions from $\omega$ to $D$. Given $s\in\Fn(\omega, D)$, define $U_s=\{f\in X^\omega :f\supseteq s\}$. Now fix $d\in D$ and let $g\in X^\omega$ be the constant function with value $d$. It is easy to see that $\langle U_{g\upharpoonright n}:n\in\omega\rangle$ is a sequence of open sets in $X^\omega$ that converges to $g$. Furthermore $\BBB=\{U_s:s\in\Fn(\omega, D)\}$ is a $\pi$-base for $X^\omega$ consisting of clopen sets that are homeomorphic to $X^\omega$. So $X^\omega$ is h-homogeneous by Proposition \ref{matveev}.
\end{proof}

\section{Infinite powers of zero-dimensional first-countable spaces}

\begin{definition}\label{homogeneous}
A space $X$ is \emph{homogeneous} if for every $x,y\in X$ there exists a homeomorphism $f:X\longrightarrow X$ such that $f(x)=y$.
\end{definition}

It is well-known (and easy to prove) that every zero-dimensional first-countable h-homogeneous space is homogeneous. As announced by Motorov (see Theorem 0.1 in \cite{motorov2}), the converse holds for zero-dimensional first-countable compact spaces of uncountable cellularity (see Theorem 2.5 in \cite{shelah} for a proof). In \cite{vandouwen}, Van Douwen constructed a zero-dimensional first-countable compact homogeneous space $X$ that is not h-homogeneous (actually, $X$ has no proper subspaces that are homeomorphic to $X$). In \cite{motorov3}, using similar techniques, Motorov constructed a zero-dimensional first-countable compact homogeneous space that is not divisible by $2$ (in the sense of Definition \ref{divisible}); see also Theorem 7.7 in \cite{vanmill2}.

In \cite{terada}, Terada asked whether $X^\omega$ is h-homogeneous for every zero-dimensional first-countable $X$. In \cite{dowpearl}, the following remarkable theorem is proved.

\begin{theorem}[Dow and Pearl]\label{dowpearl}
If $X$ is a zero-dimensional first-countable space then $X^\omega$ is homogeneous.
\end{theorem}

However, Terada's question remains open. In \cite{motorov1} and \cite{motorov2}, Motorov asks whether such an infinite power is always divisible by $2$. Using Theorem \ref{dowpearl}, we will show that the two questions are equivalent: actually even weaker conditions suffice (see Proposition \ref{hhomogeneouspowers}).

\begin{definition}\label{divisible} A space $F$ is a \emph{factor} of $X$ (or $X$ is \emph{divisible} by $F$) if there exists $Y$ such that $F\times Y\cong X$. If $F\times X\cong X$ then $F$ is a \emph{strong factor} of $X$ (or $X$ is \emph{strongly divisible} by $F$).
\end{definition}

We will use the following lemma freely in the rest of this section.

\begin{lemma} The following are equivalent.
\begin{enumerate}
\item\label{factor} $F$ is a factor of $X^\omega$.
\item\label{strongfactor} $F$ is a strong factor of $X^\omega$.
\item\label{infinitestrongfactor} $F^\omega$ is a strong factor of $X^\omega$.
\end{enumerate}
\end{lemma}

\begin{proof}
The implications $\ref{strongfactor}\rightarrow\ref{factor}$ and $\ref{infinitestrongfactor}\rightarrow\ref{factor}$ are clear.

Assume \ref{factor}. Then there exists $Y$ such that $F\times Y\cong X^\omega$, hence
$$
X^\omega\cong (X^\omega)^\omega\cong (F\times Y)^\omega\cong F^\omega\times Y^\omega.
$$
Since multiplication by $F$ or by $F^\omega$ does not change the right-hand side, it follows that \ref{strongfactor} and \ref{infinitestrongfactor} hold.
\end{proof}

\begin{lemma}\label{isolatedpoint}
Assume that $Y$ is a non-empty zero-dimensional first-countable space. Then $X=(Y\oplus 1)^\omega$ is h-homogeneous and $X\cong Y^\omega\times (Y\oplus 1)^\omega\cong  2^\omega\times Y^\omega$.
\end{lemma}

\begin{proof}
Recall that $1=\{0\}$ and let $g\in X$ be the constant function with value $0$. For each $n\in\omega$, define
$$
U_n=\{f\in X: f(i)=0\textrm{ for all }i<n\}.
$$
Observe that $\BBB=\{U_n: n\in\omega\}$ is a local base for $X$ at $g$ consisting of clopen sets that are homeomorphic to $X$. But $X$ is homogeneous by Theorem \ref{dowpearl}, therefore it has such a local base at every point. In conclusion $X$ has a base (hence a $\pi$-base) consisting of clopen sets that are homeomorphic to $X$. It follows from Proposition \ref{matveev} that $X$ is h-homogeneous.

To prove the second statement, observe that
$$
X\cong (Y\oplus 1)\times X\cong (Y\times X)\oplus X,
$$
hence $X\cong Y\times X$ by h-homogeneity. It follows that $X\cong Y^\omega\times (Y\oplus 1)^\omega$. Finally,
$$
Y^\omega\times (Y\oplus 1)^\omega\cong \left(Y^\omega\times (Y\oplus 1)\right)^\omega\cong (Y^\omega\oplus Y^\omega)^\omega\cong 2^\omega\times Y^\omega,
$$
that concludes the proof.
\end{proof}

\begin{proposition}\label{hhomogeneouspowers} Assume that $X$ is a zero-dimensional first-countable space containing at least two points. Then the following are equivalent.
\begin{enumerate}
\item\label{xoplus1} $X^\omega\cong (X\oplus 1)^\omega$.
\item\label{yisolated} $X^\omega\cong Y^\omega$ for some space $Y$ with at least one isolated point.
\item\label{hhomogeneous} $X^\omega$ is h-homogeneous.
\item\label{divisiblebytwo} $X^\omega$ has a non-empty clopen subset that is strongly divisible by $2$.
\item\label{properclopen} $X^\omega$ has a proper clopen subset that is homeomorphic to $X^\omega$.
\item\label{properclopenfactor} $X^\omega$ has a proper clopen subset that is a factor of $X^\omega$.
\end{enumerate}
\end{proposition}

\begin{proof}
The implication $\ref{xoplus1}\rightarrow\ref{yisolated}$ is trivial; the implication $\ref{yisolated}\rightarrow\ref{hhomogeneous}$ follows from Lemma \ref{isolatedpoint}; the implications $\ref{hhomogeneous}\rightarrow\ref{divisiblebytwo}\rightarrow\ref{properclopen}\rightarrow\ref{properclopenfactor}$ are trivial.

Assume that $\ref{properclopenfactor}$ holds. Let $C$ be a proper clopen subset of $X^\omega$ that is a factor of $X^\omega$ and let $D=X^\omega\setminus C$. Then
\begin{eqnarray}\nonumber
X^\omega & \cong & (C\oplus D)\times X^\omega\\\nonumber
& \cong & (C\times X^\omega) \oplus (D\times X^\omega)\\\nonumber
& \cong & X^\omega\oplus (D\times X^\omega)\\\nonumber
& \cong & (1\oplus D)\times X^\omega,
\end{eqnarray}
hence $X^\omega\cong (1\oplus D)^\omega\times X^\omega$. Since $(1\oplus D)^\omega\cong 2^\omega\times D^\omega$ by Lemma \ref{isolatedpoint}, it follows that $2^\omega$ is a factor of $X^\omega$. So $2^\omega$ is a strong factor of $X^\omega$. Therefore \ref{xoplus1} holds by Lemma \ref{isolatedpoint}.
\end{proof}

The next two propositions show that in the pseudocompact case we can say something more.

\begin{proposition}\label{properclopenpseudocompact}
Assume that $X$ is a zero-dimensional first-countable space such that $X^\omega$ is pseudocompact. Then $C^\omega\cong (X\oplus 1)^\omega$ for every non-empty proper clopen subset $C$ of $X^\omega$.
\end{proposition}

\begin{proof}
Let $C$ be a non-empty proper clopen subset of $X^\omega$. It follows from Corollary \ref{finitelymanycoordinates} that $C\cong C\times X^\omega$, hence $C^\omega\cong C^\omega\times X^\omega$. Since $C^\omega\times X^\omega$ clearly has a proper clopen subset that is homeomorphic to $C^\omega\times X^\omega$, Proposition \ref{hhomogeneouspowers} implies that $C^\omega$ is h-homogeneous, hence strongly divisible by $2$. So $C^\omega\cong 2^\omega\times C^\omega\cong 2^\omega\times C^\omega\times X^\omega$. Since $2^\omega\times X^\omega\cong (X\oplus 1)^\omega$ by Lemma \ref{isolatedpoint}, it follows that $C^\omega\cong C^\omega\times (X\oplus 1)^\omega$.

On the other hand, $(X\oplus 1)^\omega\cong X^\omega\times (X\oplus 1)^\omega$ by Lemma \ref{isolatedpoint}. Hence $(X\oplus 1)^\omega$ has a clopen subset homeomorphic to $C\times (X\oplus 1)^\omega$. But Lemma \ref{isolatedpoint} shows that $(X\oplus 1)^\omega$ is h-homogeneous, so $C\times (X\oplus 1)^\omega\cong (X\oplus 1)^\omega$. Therefore $C^\omega\times (X\oplus 1)^\omega\cong (X\oplus 1)^\omega$, that concludes the proof.
\end{proof}

\begin{proposition} In addition to the hypotheses of Proposition \ref{hhomogeneouspowers}, assume that $X^\omega$ is pseudocompact. Then the following can be added to the list of equivalent conditions.
\begin{enumerate}
\item[(7)] $X^\omega$ has a non-empty proper clopen subset that is homeomorphic to $Y^\omega$ for some space $Y$.
\end{enumerate}
\end{proposition}

\begin{proof}
The implication $5\rightarrow 7$ is trivial.

Assume that 7 holds. Let $C$ be a non-empty proper clopen subset of $X^\omega$ that is homeomorphic to $Y^\omega$ for some space $Y$. Then clearly $C^\omega\cong C$. Therefore $C\cong (X\oplus 1)^\omega$ by Proposition \ref{properclopenpseudocompact}. Hence $C$ is strongly divisible by $2$ by Lemma \ref{isolatedpoint}, showing that 4 holds.
\end{proof}

Finally, we point out that Proposition \ref{hhomogeneouspowers} can be used to give a positive answer to Terada's question for a certain class of spaces. We will need the following definition.

\begin{definition} A space $X$ is \emph{ultraparacompact} if every open cover of $X$ has a refinement consisting of pairwise disjoint clopen sets.
\end{definition}

It is easy to see that every ultraparacompact space is zero-dimensional. As noted by Nyikos in \cite{nyikos}, a space is ultraparacompact if and only if it is paracompact and strongly zero-dimensional (this is proved like Proposition 1.2 in \cite{ellis}). A metric space $X$ is ultraparacompact if and only if $\dim X=0$ (see Theorem 7.2.4 in \cite{engelking}); see also Theorem 7.3.3 in \cite{engelking}. For such a metric space $X$, Van Engelen proved that $X^\omega$ is h-homogeneous if $X$ is of the first category in itself or $X$ has a completely metrizable dense subset (see Theorem 4.2 and Theorem 4.4 in \cite{vanengelen}). It follows that $X^\omega$ is h-homogeneous if $X$ is analytic (see Corollary \ref{dichotomy}). For related results, see also Theorem 8 and Theorem 9 in \cite{ostrovskii}.

\begin{proposition}
Assume that $X$ is a (zero-dimensional) first-countable space. If $X^\omega$ is ultraparacompact and non-Lindel\"of then $X^\omega$ is h-homogeneous.
\end{proposition}
\begin{proof}
Let $\UUU$ be an open cover of $X^\omega$ with no countable subcovers. By ultraparacompactness, there exists a refinement $\VVV$ of $\UUU$ consisting of pairwise disjoint non-empty clopen sets. Let $\VVV=\{C_\alpha:\alpha\in\kappa\}$ be an enumeration without repetitions, where $\kappa$ is an uncountable cardinal.

Now fix $x\in X^\omega$ and a local base $\{U_n:n\in\omega\}$ at $x$ consisting of clopen sets. Since $X^\omega$ is homogeneous by Theorem \ref{dowpearl}, for each $\alpha<\kappa$ we can find $n(\alpha)\in\omega$ such that a homeomorphic clopen copy $D_\alpha$ of $U_{n(\alpha)}$ is contained in $C_\alpha$. Since $\kappa$ is uncountable, there exists an infinite $S\subseteq\kappa$ such that $n(\alpha)=n(\beta)$ for every $\alpha,\beta\in S$. It is easy to check that $\bigcup_{\alpha\in S}D_\alpha$ is a non-empty clopen subset of $X^\omega$ that is strongly divisible by $2$. Therefore $X^\omega$ is h-homogeneous by Proposition \ref{hhomogeneouspowers}. 
\end{proof}

An application of Corollary 4.1.16, Theorem 7.3.2 and Theorem 7.3.16 in \cite{engelking} immediately yields the following.

\begin{corollary} Assume that $X$ is a metric space such that $\dim X=0$. If $X$ is non-separable then $X^\omega$ is h-homogeneous.
\end{corollary}

\noindent\textbf{Acknowledgements.} The author thanks his advisor Ken Kunen and the anonymous referee for valuable comments on earlier versions of this paper.

\appendix

\section{Proofs of the results by Terada and Matveev}

In this section we will present a somewhat unified approach to the proofs of Theorem \ref{teradanonpseudocompact} and Proposition \ref{matveev}. Notice that zero-dimensionality is never needed.

\begin{proof}[Proof of Theorem \ref{teradanonpseudocompact}]
Assume that $X$ has a $\pi$-base consisting of clopen sets that are homeomorphic to $X$. By Lemma \ref{dispensezerodimensionality}, we can fix a collection $\{X_n:n\in\omega\}$ consisting of pairwise disjoint non-empty clopen subsets of $X$ such that $X=\bigcup_{n\in\omega}X_n$. Let $C$ be a non-empty clopen subset of $X$. Since $C$ contains a clopen subset that is homeomorphic to $X$, we can fix a collection $\{C_n:n\in\omega\}$ consisting of pairwise disjoint non-empty clopen subsets of $C$ such that $C=\bigcup_{n\in\omega}C_n$.

We will recursively construct clopen sets $Y_n\subseteq X_n$ and $D_n\subseteq C_n$, together with partial homeomorphisms $h_n$ and $k_n$ for every $n\in\omega$. In the end, setting $h=\bigcup_{n\in \omega}(h_n\cup k_n)$ will yield the desired homeomorphism. Start by setting $Y_0=\varnothing$ and $h_0=\varnothing$. Then, let $D_0\subseteq C_0$ be a clopen set that is homeomorphic to $X_0\setminus Y_0$ and fix a homeomorphism $k_0:X_0\setminus Y_0\longrightarrow D_0$. Now assume that clopen sets $D_n\subseteq C_n$ and $Y_n\subseteq X_n$ have been defined. Let $Y_{n+1}\subseteq X_{n+1}$ be a clopen set that is homeomorphic to $C_n\setminus D_n$ and fix a homeomorphism $h_{n+1}:Y_{n+1}\longrightarrow C_n\setminus D_n$. Then, let $D_{n+1}\subseteq C_{n+1}$ be a clopen set that is homeomorphic to $X_{n+1}\setminus Y_{n+1}$ and fix a homeomorphism $k_{n+1}:X_{n+1}\setminus Y_{n+1}\longrightarrow D_{n+1}$.
\end{proof}

\begin{proof}[Proof of Proposition \ref{matveev}]
Let $\langle U_n:n\in\omega\rangle$ be a sequence of non-empty open subsets of $X$ that converges to a point $x$. One can easily obtain a sequence $\langle X_n:n\in\omega\rangle$ of pairwise disjoint non-empty clopen sets that converges to $x$, such that $x\notin X_n$ for each $n\in\omega$. Let $C$ be a non-empty clopen subset of $X$. Let $B$ be a clopen subset of $C$ that is homeomorphic to $X$. Fix a homeomorphism $f:X\longrightarrow B$ and let $C_n=f[X_n]$ for each $n\in\omega$.

Now define clopen sets $Y_n\subseteq X_n$ and $D_n\subseteq C_n$ for each $n\in\omega$ and a (partial) homeomorphism $h$ as in the proof of Theorem \ref{teradanonpseudocompact}, but start by choosing $Y_0$ homeomorphic to $C\setminus B$ and fixing a homeomorphism $h_0:Y_0\longrightarrow C\setminus B$. Finally, extend $h$ by setting $h(x)=f(x)$ for every $x\in X\setminus\bigcup_{n\in\omega}X_n$. It is easy to check that this yields the desired homeomorphism.
\end{proof}

\section{Some Descriptive Set Theory}

The following results seem to be folklore, but we could not find satisfactory references. For the definitions of \emph{analytic} and \emph{property of Baire}, see \cite{vanmill3}.

\begin{theorem}
Let $X$ be an analytic metric space. Then either $X$ has a completely metrizable dense subset or $X$ has a non-empty open subset of the first category.
\end{theorem}

\begin{proof}
Let $\widetilde{X}$ be the completion of $X$. By Theorem A.13.13 in \cite{vanmill3}, $X$ has the property of Baire in $\widetilde{X}$. Therefore, by Proposition A.13.10 in \cite{vanmill3}, we can write $X=G\cup M$, where $G$ is a $G_\delta$ subset of $\widetilde{X}$ and $M$ is of the first category in $\widetilde{X}$.

Since $G$ is a $G_\delta$ subset of the complete metric space $\widetilde{X}$, it is completely metrizable (see Theorem A.6.3 in \cite{vanmill3}). Since $X$ is dense in $\widetilde{X}$, the set $M$ is of the first category in $X$ as well (see Exercise A.13.7 in \cite{vanmill3}). In conclusion, if $G$ is dense in $X$ then the first alternative in the statement of the theorem will hold, otherwise the second alternative will hold.
\end{proof}

\begin{corollary}\label{dichotomy}
Let $X$ be an analytic metric space. Then either $X^\omega$ has a completely metrizable dense subset or $X^\omega$ is of the first category in itself.
\end{corollary}

\begin{proof}
If $X$ has a completely metrizable dense subset $D$ then $D^\omega$ is a completely metrizable dense subset of $X^\omega$ (see Lemma A.6.2 in \cite{vanmill3}).

So assume that $X$ has a non-empty open subset $U$ of the first category. Observe that $M_n=\{f\in X^\omega:f(n)\in U\}$ is of the first category in $X^\omega$ for every $n\in\omega$. Also, it is clear that $(X\setminus U)^\omega$ is closed nowhere dense in $X^\omega$. It follows that $X^\omega$ is of the first category in itself.
\end{proof}

\end{document}